\documentclass[11pt]{amsart}
\usepackage{amsmath,amssymb, graphicx, amscd,latexsym}
\makeatletter
\newtheorem{Theorem}{Theorem}
\newtheorem{Lemma}[Theorem]{Lemma}

\newtheorem{Proposition}[Theorem]{Proposition}\newtheorem{Claim}[Theorem]{Claim}

\newtheorem{Definition}[Theorem]{Definition}
\newtheorem{Remark}[Theorem]{Remark}

\newtheorem{Example}[Theorem]{Example}

\newtheorem{Problem}[Theorem]{Problem}

\newtheorem{Assertion}[Theorem]{Assertion}

\newcommand{\eps}{\varepsilon}

\newcommand\la{\lambda}
\newcommand\vphi{\varphi}

\newcommand\al{\alpha}

\newcommand\si{\sigma}
\newcommand\be{\beta}

\newcommand\ga{\gamma}

\newcommand\de{\delta}
\newcommand\De{\Delta}

\newcommand\BC{ {\mathbb C}}
\newcommand\BN{ {\mathbb  N}}

\newcommand\BR{ {\mathbb  R}}
\newcommand\BP{ {\mathbb  P}}

\newcommand\bfZ{\mbox {\bf  Z}}

\newcommand\bfu{\mbox {\bf  u}}

\newcommand\bfv{\mbox {\bf  v}}

\newcommand\bfx{\mbox {\bf  x}}
\newcommand\bfX{\mbox {\bf  X}}
\newcommand\bfw{\mbox {\bf  w}}

\newcommand\bfz{\mbox {\bf  z}}

\newcommand\bfa{\mbox {\bf  a}}

\newcommand\nin{\noindent}

\newcommand\nl{\newline}

\newcommand\grad{{\rm{grad}\/}}
\newcommand\tgrad{{{}^t\rm{grad}\/}}

\newcommand\id{\rm{id}}

\newcommand\Sign{{\rm{Sign}}}

\newcommand \sm{\rm{m_s}\/}
\newcommand \SM{\rm{m_{s,tot}}\/}

\newcommand\inv{^{-1}}

\def\mapright#1{\smash{\mathop{\longrightarrow}\limits^{{#1}}}}

\def\mapdown#1{\Big\downarrow\rlap{$\vcenter{\hbox{$#1$}}$}}

\def\inv{^{-1}}

\begin{document}
\title[ Intersection Theory on Mixed Curves
]
{ Intersection Theory on Mixed Curves
}

\author
[M. Oka ]
{Mutsuo Oka }
\address{\vtop{
\hbox{Department of Mathematics}
\hbox{Tokyo  University of Science}
\hbox{1-3 Kagurazaka, Shinjuku-ku}
\hbox{Tokyo 162-8601}
\hbox{\rm{E-mail}: {\rm oka@rs.kagu.tus.ac.jp}}
}}

\begin{abstract}
We consider two mixed curve  $C,C'\subset \BC^2$ which are defined by 
mixed functions of two variables $\bfz=(z_1,z_2)$.
We have shown in  \cite{MC},  that they have canonical orientations.
 If $C$ and $C'$ are smooth and intersect transversely at $P$,
the  
 intersection number  $I_{top}(C,C';P)$ is  topologically defined.
We will generalize this definition to the case when
 the intersection is not  necessarily transversal or either  $C$ or $C'$ may be singular at $P$
using the defining mixed polynomials.
\end{abstract}
\maketitle

\maketitle

\section{Introduction}
First we recall the complex analytic situation.
Consider  complex polynomials $f(\bfz)$ and $g(\bfz)$
of two variables $\bfz=(z_1,z_2)$ and consider
 complex analytic  curves defined by 
$C: f(\bfz)=0$ and $C': g(\bfz)=0$.
Suppose that 
$P$ is an isolated  intersection point of $C\cap C'$.
 Then the local algebraic intersection number
$I(f,g;P)$ is defined by the dimension of the quotient module
$\dim\,\mathcal O_P/(f,g)$ where $\mathcal O_P$ is the local ring of the
holomorphic functions at $P$
and $(f,g)$ is the ideal generated by $f$ and $g$.
Thus  $I(f,g;P)$  a strictly positive integer
 and it is equal to 1 if and only if 
$C$ and $C'$ are non-singular at $P$ and transversal each other.
 On the other hand, the complex curves $C,C'$ have  canonical orientations 
which come from their complex structures  
(see for example, \cite{GriffithsEtc}) and   the local algebraic
intersection number is equal to   the local topological intersection
number
if the intersection is transverse. Moreover this is also true for
non-transverse intersection in the sense that 
under a slight perturbation,
an intersection $P$ of algebraic intersection number $\nu$ splits
into $\nu$ transverse intersections.
In particular, the topological local  intersection number
can  be defined by the algebraic local intersection number. 

The purpose of this note is to define  the  local
intersection multiplicity for  two mixed curves using the defining polynomials
and study the analogues properties.
The problem in this case is that the local intersection number is not necessarily positive.
This makes the algebraic calculation more difficult.
Let $C: \,f(\bfz,\bar \bfz)=0$ and $C':\,g(\bfz,\bar\bfz)=0$
be mixed curves which have at worst isolated mixed singularity at $P\in
C\cap C'$. We will define the intersection multiplicity
$I_{top}(C,C';P)$
using a certain mapping degree which is described by the defining polynomials $f,g$
(Definition \ref{local intersection}, \S 2 and Theorem \ref{Theorem}). 
This definition coincides with the usual one for   complex analytic curves.

In \S 4,
we consider the roots of  a mixed polynomial $h(u,\bar u)$ of one variable $u$ as a special case.
We introduce the notion of {\em multiplicity with sign}  $\sm(f,\al)$ for a root $\al$ of $h(u,\bar u)=0$
and  we give  a formulae for the description of $\sm(f,\al)$ for  an
admissible mixed  polynomial $h(u,\bar u)$
(Theorem \ref{main}).

\section{Mixed curves}
\subsection{A mixed singular point}
Let $f(\bfz,\bar\bfz),\,\bfz=(z_1,z_2)\in \BC^2$, be a mixed polynomial.
See \cite{OkaMix, MC} for further details about   a mixed polynomial.
Using real coordinates $(x_1,y_1,x_2,y_2)$ with $z_j=x_j+iy_j,\,j=1,2$,
$f$ can be understood as
 a sum of two polynomials with real coefficients:
\[
 f(\bfz,\bar\bfz)= f_{\BR}(x_1,y_1,x_2,y_2)+ i f_I(x_1,y_1,x_2, y_2).
\]
where $f_{\BR},f_I$ are  the real part and the  imaginary part of $f$
respectively.
Recall that $f(\bfz,\bar \bfz)$ is a polynomial of 
$x_1,y_1,x_2,y_2$ by the substitution
\[
 z_j=\frac{z_j+\bar z_j}2,\quad
\bar z_j=\frac{z_j-\bar z_j}{2i},\quad j=1,2.
\]
We say that   $C: f(\bfz,\bar\bfz)=0$ is  {\em  mixed non-singular at 
$P\in C$ }
if the Jacobian matrix of $(f_{\BR}, f_I)$ has rank two at $P$
(\cite{OkaPolar, OkaMix}).
We recall that $\BC^2$ has a canonical orientation given from the complex structure.
We identify $\BC^2$ with $\BR^4$ by $(z_1,z_2)\leftrightarrow (x_1,y_1,x_2,y_2)$ and thus  a positive frame of $\BR^4$ is  given by $(\frac{\partial}{\partial x_1},\frac{\partial}{\partial y_1},\frac{\partial}{\partial x_2},\frac{\partial}{\partial y_2})$.
If $P$ is a mixed non-singular point, $C$ is locally
a real two dimensional manifold.
The normal bundle $N_{C,P}$ of $C\subset \BC^2$ at $P$ has the canonical orientation 
so that the orientation is compatible with the complex valued function
$f$, namely $df_P:N_{C,P}\to T_{0}\BC$ is an orientation preserving isomorphism.
Thus the orientation of $C$ at $P$ is defined as follows.  A frame $(\bfv_1,\bfv_2)\subset T_P C$,
$\bfv_1=(v_{11},v_{12},v_{13},v_{14}),\bfv_2=(v_{21},v_{22},v_{23},v_{24})$, is positive  if and only if
the frame
\[
 M:=\left(\begin{matrix}\bf\bfv_1\\\bf\bfv_2\\{\grad}\,f_{\BR}\\{\grad}\,f_I\end{matrix}\right)=\left(
\begin{matrix}
v_{11}&v_{12}&v_{13}&v_{14}\\
v_{21}&v_{22}&v_{23}&v_{24}\\
\frac{\partial f_{\BR}}{\partial x_1}&\frac{\partial f_{\BR}}{\partial y_1}&
\frac{\partial f_{\BR}}{\partial x_2}&\frac{\partial f_{\BR}}{\partial y_2}\\
\frac{\partial f_I}{\partial x_1}&
\frac{\partial f_I}{\partial y_1}&
\frac{\partial f_I}{\partial x_2}&
\frac{\partial f_I}{\partial y_2}\\
\end{matrix}
\right)
\]
is a positive frame of $\BC^2=\BR^4$. 
 The  
gradient
vector ${\grad}\, h(x_1,y_1,x_2,y_2)$ of a real valued function $h$ 
is defined by 
\[
{\grad}\, h(x_1,y_1,x_2,y_2)= (\frac{\partial h}{\partial x_1},\frac{\partial h}{\partial y_1},\frac{\partial h}{\partial x_2},\frac{\partial
h}{\partial y_2}).
\]


\subsection{Mixed homogenization and the closure in $\BP^2$} Assume that $f(\bfz,\bar\bfz)=\sum_{\nu,\mu}c_{\nu\mu} \bfz^{\nu}{\bar\bfz}^\mu$ is a
mixed polynomial
of two variables $\bfz=(z_1,z_2)$.
Put $C=f\inv(0)\subset \BC^2$.
We assume that $C$ is non-empty  and that  $C$ has only finite number of mixed
singular points.
 We consider the affine space $\BC^2$ 
with coordinates $\bfz$ as  the affine chart $Z_0\ne 0$ of the
projective space $\BP^2$ with homogeneous coordinates $(Z_0,Z_1,Z_2)$.
The coordinates are related by 
$z_1=Z_1/Z_0,z_2=Z_2/Z_0$.
Let $d^+$ and $d^-$ be the degree of $f(\bfz,\bar \bfz)$ in $\bfz$ and
$\bar \bfz$ respectively.
That is,
\[
 d^+=\max\{|\nu|\,|\, c_{\nu\mu}\ne 0\},\,\,d^-=\max\{|\mu|\,|\, c_{\nu\mu}\ne 0\}
\]
where  $|\nu|=\nu_1+\nu_2  $ for a multi-integer $\nu=(\nu_1,\nu_2  )$.
We assocate with $f$ a strongly  polar homogeneous mixed polynomial $F(\bfZ,\bar\bfZ)$ as follows, where 
$\bfZ=(Z_0,Z_1,Z_2)$
and $\bar \bfZ=(\bar Z_0,\bar Z_1,\bar Z_2)$ by
$ F(\bfZ,\bar \bfZ):=Z_0^{d_+}{\bar Z_0}^{d^-}f(\frac{Z_1}{Z_0},\frac{Z_2}{Z_0},
 \frac{\bar Z_1}{\bar Z_0},\frac{\bar Z_2}{\bar Z_0})$
and we call $F$ {\em the mixed homogenization of $f$}.
We define  $\bar C\subset \BP^2$ by the topological closure
of $C\subset \BC^2\subset \BP^2$ and 
we defines a mixed  projective curve $\tilde C:=\{((Z_o:Z_1:Z_2)\in \BP^2\,|\,F(\bfZ,\bar\bfZ)=0\}$.
It is easy to see that 
the closure  
$\bar C$ of $C$ in $\BP^2$ is a subset of $\tilde C$ but in general, $\bar C$ might be a proper  subvariety of $\tilde C$.
 $F$ is a strongly polar homogeneous polynomial of radial degree
$d^++d^-$
and the polar degree $d^+-d^-$respectively
and $F|_{Z_0\ne 0}=f$.  In  \cite{MC}, we have assumed that  the  polar degree
is non-zero for the definition of strongly
polar homogeneous polynomials, but in this paper, we consider also the case $d^+=d^-$.

\section{Intersection numbers}
\subsection{Local intersection number I (Smooth and transversal intersection case)}
In this section, we denote vectors in $\BR^4$ by {\bf column vectors} for  brevity's  sake.
 Assume that 
$C:f=0$ and $C':g=0$ are two mixed curves and assume that 
$P\in C\cap C'$ and $C,C'$ are mixed non-singular at $P$ and  the intersection  is transverse at
$P$.
Let $\bfu_1,\bfu_2 $ and $\bfv_1,\bfv_2$ be positive frames of $T_P C$ and $T_P C'$.
Then the local (topological)
 intersection number $I_{top}(C,C';P)$ is  defined by the sign of the determinant
$\det(\bfu_1,\bfu_2 ,\bfv_1,\bfv_2)$ (See for example \cite{Milnor-h-cobordism}).  Namely
\[
I_{top}(C,C';P)=\begin{cases}
1,&\quad  \det(\bfu_1,\bfu_2 ,\bfv_1,\bfv_2)>0,\\
-1,&\quad \det(\bfu_1,\bfu_2 ,\bfv_1,\bfv_2)<0.
\end{cases}
\] 
For any frames  $\bfw_1,\dots,\bfw_4$ of $\BR^4$, we define 
\[
\Sign (\bfw_1,\bfw_2,\bfw_3,\bfw_4):=\begin{cases}
1,\quad &\text{if}\,\,\det(\bfw_1,\bfw_2,\bfw_3,\bfw_4)>0\\
-1,\quad &\text{if}\,\,\det(\bfw_1,\bfw_2,\bfw_3,\bfw_4)<0.
\end{cases}
\]
By the definition of the orientation of $C$ and $C'$,
\[
\begin{split}
& \Sign (\bfu_1,\bfu_2 ,{\tgrad}\,f_{\BR}(P),{\tgrad}\,f_I(P))=1,\\
&
\Sign (\bfv_1,\bfv_2,{\tgrad}\, g_{\BR}(P),{\tgrad}\, g_I(P))=1.
\end{split}
\]
Now our first result is the following.
\begin{Theorem}\label{Theorem}
The intersection number $ I_{top}(C,C';P)$ is given by 
\[
 \Sign  ({\tgrad}\,f_{\BR}(P),{\tgrad}\,f_I(P),{\tgrad}\,
 g_{\BR}(P),{\tgrad}\, g_I(P)).
\]
 \end{Theorem}
Recall that the tangent space
$T_PC$ is generated by  the vectors orthogonal to
the two dimensional subspace  $<{\tgrad}\, f_{\BR}(P),{\tgrad}\,f_I(P)>_{\BR}$.
Thus
 two dimensional planes $<\bfu_1,\bfu_2 >_{\BR}$ and $<{\tgrad}\, f_{\BR}(P),{\tgrad}\,f_I(P)>_{\BR}$  are
orthogonal.
Here $<\bfw_1,\bfw_2>_{\BR}$ is the   two dimensional plane spanned by $\bfw_1,\bfw_2$.
\subsubsection{Gram-Schmidt orthonormalization}
First we consider a simple assertion.
Let $\bfa_1,\bfa_2,\bfa_3,\bfa_4$ be column vectors in $\BR^4$
and let $P,Q$ be $2\times 2$
matrices.
Then
\begin{Assertion}\label{Assertion}
\[
\det((\bfa_1,\bfa_2)P,(\bfa_3,\bfa_4)Q)=\det(\bfa_1,\bfa_2,\bfa_3,\bfa_4)\det(P)\det(Q)
\]
\end{Assertion}
\begin{proof} The assertion follows from the simple equality in $4\times 4$ matirices:
\[
((\bfa_1,\bfa_2)P,(\bfa_3,\bfa_4)Q)=(\bfa_1,\bfa_2,\bfa_3,\bfa_4)
\left(\begin{matrix} P&O\\O&Q
\end{matrix}
\right).
\]
\end{proof}
Now we consider Gram-Schmidt   orthnormalization
 of $(\bfu_1,\bfu_2 )$ and \nl
$({\tgrad}\, f_{\BR},{\tgrad}\,f_I)$. They are orthonormal frames
$(\bfu_1',\bfu_2')$ and \nl $({\tgrad}\, f_{\BR}(P)',{\tgrad}\,f_I(P)')$
such that they satisfy the equalities:
\[
\begin{split}
&(\bfu_1',\bfu_2')= (\bfu_1,\bfu_2 )Q_1,\,\,\text{and} \\
&({{\tgrad}}\, f_{\BR}(P)',{{\tgrad}}\,f_I(P)')=({{\tgrad}}\, f_{\BR}(P),{\tgrad}\,f_I(P))Q_2
\end{split}
\] 
where $Q_1,Q_2$ are upper triangular  $2\times 2$ matrices with positive entries in their diagonals.
Similarly we consider the orthnormalization
\[\begin{split}
& (\bfv_1',\bfv_2')=(\bfv_1,\bfv_2)R_1\\
&({\tgrad} g_{\BR}(P)',{\tgrad} g_I(P)')=
 ({\tgrad} g_{\BR}(P),{\tgrad} g_I(P)) R_2,\end{split}
\]
where $ R_1,R_2$ are upper triangular matrices with
with positive entries  in their diagonals.
Using  Assertion \ref{Assertion}, we get 
\[
\begin{split}
\Sign(\bfu_1',\bfu_2',\bfv_1',\bfv_2')&=
\Sign((\bfu_1,\bfu_2)Q_1,(\bfv_1,\bfv_2)R_1)\\
&=\Sign(\bfu_1,\bfu_2,\bfv_1,\bfv_2)
\end{split}
\]
\[
\begin{split}
\Sign &({\tgrad} f_{\BR}(P)',{\tgrad}\,f_I(P)',{\tgrad}\,
  g_{\BR}(P)',{\tgrad}\,g_I(P)')\\
&=\Sign(({\tgrad} f_{\BR}(P),{\tgrad}\,f_I(P))Q_2,({\tgrad}\,g_{\BR}(P),
{\tgrad}\,g_I(P))R_2)\\
&= \Sign({\tgrad} f_{\BR}(P),{\tgrad}\,f_I(P),{\tgrad}\,g_{\BR}(P),{\tgrad}\,g_I(P)).
\end{split}
\]
Thus
 the  calculation of the intersection number  can be done using these
 orthonormal frames
\[
 (\bfu_1',\bfu_2',{\tgrad}\, f_{\BR}(P)',{\tgrad}\,f_I(P)'),\,
(\bfu_1',\bfv_2',{\tgrad}\, g_{\BR}(P)',{\tgrad}\,g_I(P)').
\]

Thus the proof of Theorem \ref{Theorem} reduces to the following.
\begin{Lemma} Assume that $(\bfu_1,\bfu_2 ,\bfu_3,\bfu_4)$ and $(\bfv_1,\bfv_2,\bfv_3,\bfv_4)$ be positive orthonormal frames
of $\BR^4$.
Then 
\[
\det (\bfu_1,\bfu_2 ,\bfv_1,\bfv_2)=\det (\bfu_3,\bfu_4,\bfv_3,\bfv_4).
\]
\end{Lemma}
\begin{proof}

  Assume that 
  \begin{eqnarray}\label{trans1}
  (\bfu_1,\bfu_2
 ,\bfu_3,\bfu_4)=(\bfv_1,\bfv_2,\bfv_3,\bfv_4)A,
 \end{eqnarray}
with 
 $A\in SO(4;\BR)$. Write $A$ by $2\times 2$ matrices  as
\[
 A=\left(\begin{matrix}
A_1& A_2\\
B_1&B_2
\end{matrix}
\right)
\]
The equality  (\ref{trans1}) can be rewritten as
\begin{eqnarray}\label{trans2}
(\bfv_1,\bfv_2,\bfv_3,\bfv_4)=(\bfu_1,\bfu_2 ,\bfu_3,\bfu_4)\,{}^tA
\end{eqnarray}
where
\[
{}^t A=\left(\begin{matrix}
{}^tA_1& {}^tB_1\\
{}^t A_2&{}^tB_2
\end{matrix}
\right).
\]
First we consider the equality from (\ref{trans1}):
\[\begin{split}
 \det(\bfu_1,\bfu_2,\bfv_1,\bfv_2)&=
\det((\bfv_1,\bfv_2)A_1+(\bfv_3,\bfv_4)B_1,\bfv_1,\bfv_2)\\
&=\det(\bfv_1,\bfv_2,(\bfv_3,\bfv_4)B_1)\\
&=\det B_1.
\end{split}
\]
On the other hand, we have also  from (\ref{trans2}):
\begin{eqnarray}
 \det(\bfu_1,\bfu_2
		 ,\bfv_1,\bfv_2)&=\det(\bfu_1,\bfu_2,(\bfu_3,\bfu_4)\,{}^tA_2)\\
&=\det\,{}^tA_2=\det A_2
\end{eqnarray}
Thus $\det A_2=\det B_1$.
Similarly we get 
\[
\begin{split}
 \det(\bfu_3,\bfu_4,\bfv_3,\bfv_4) &=\det((\bfv_1,\bfv_2)A_2,\bfv_3,\bfv_4) =\det\, A_2 \\
 &=\det(\bfu_3,\bfu_4,(\bfv_1,\bfv_2)\, {}^tB_1)=\det\,B_1.\\
\end{split}
\]
Thus  the assertion
follows from these equalities.
\end{proof}
\subsection{Local intersection number II (General case)}
Assume that $C:\,f(\bfz,\bar\bfz)=0$ and  $C': g(\bfz,\bar \bfz)=0$ be mixed curve as above and let $P$ be an isolated
intersection
point of $C\cap C'$. We assume also that   both  $C$ and $C'$ have  at worst an isolated
mixed singularity at $P$.
\begin{Definition} \label{local intersection}Let 
$ \vphi=(f_{\BR},f_I, g_{\BR},g_I):\BR^4\to \BR^4$.
We define the local intersection number
$I_{top}C,C';P)$ by the local mapping degree of
the normalized 
mapping
$\psi$  of $\vphi$:
\[
\psi:=\vphi/\|\vphi\|:S_\eps^3(P)\to S^3.
\]
Here  $S_\eps^3(P):=\{\bfx\in \BR^4\,|\, \|\bfx-P\|=\eps\}$ and $\eps$
is a sufficiently small positive number so that $P$ is the only
 intersection of $C$ and $C'$
 in $B_\eps(P)$ where 
 $B_\eps(P)$ is the disk of radius $\eps$ centered at $P$.
\end{Definition}
Suppose that $P$ is a transverse intersection of $C$ and $C'$  and
assume that  $C$ and $C'$ are mixed
smooth  at $P$. Take a small positive number $\eps$ so that 
\[
\| \vphi^{(1)}(\bfz)\|\ge 2\|\vphi-\vphi^{(1)}(\bfz)\|,
 \quad \|\bfz-P\|=\eps
\]
where
$\vphi^{(1)}$ is the linear term of $\vphi$ at $P$. 
Then we consider the homotopy $\vphi_t=(1-t)\vphi+t\vphi^{(1)},\,0\le t\le 1$.
Then the normalized mapping $\psi$
is homotopic to that of $\vphi^{(1)}$ on $S_\eps^3(P)$. The latter is nothing but the 
normalization of $({\tgrad}\,f_{\BR},{\tgrad} \, f_{I},{\tgrad}\, g_{\BR},{\tgrad}\, g_I)$.
Thus
\begin{Proposition} This definition coincides with the topological local
intersection number if  the intersection is transverse and two curves 
$C,C'$ are mixed non-singular at $P$.
\end{Proposition}
\subsubsection{Stability of the intersection number under a bifurcation}
Consider two mixed algebraic curves $C: f=0$ and $C':g=0$ and assume that 
$P\in C\cap C'$ be an isolated point of $C\cap C'$ (but probably not a
transversal intersection).
Let $f_t,g_t,\,|t|\le \rho$ be two continuous families of mixed polynomials
such that $f_0=f,\,g_0=g$.
We take  a fixed $\eps>0$  so that $C\cap C'\cap B_\eps^4(P)=\{P\}$
with
$B_\eps^4(P)=\{\bfz\in \BC^2\,|\, \|\bfz-P\|\le \eps\}$ and put 
$C_t=\{\bfz\in \BC^2\,|\, f_t(\bfz,\bar \bfz)=0\}$
and 
$C'_s=\{\bfz\in \BC^2\,|\, g_s(\bfz,\bar\bfz)=0\}$.
Take a sufficiently small $\gamma>0$ so that 
\[
\{\bfz\in S_\eps\,|\,f_\al(\bfz,\bar\bfz)=g_\be(\bfz,\bar \bfz)=0\}=\emptyset,
\quad |\al|,\be|\le \ga\le \rho.
\]
Take $\de,\de'$  with $|\de|,|\de'|\le \ga$ and assume that 
$C_\de\cap C'_{\de'}\cap B_\eps^4(P)=\{P_1,\dots, P_\nu\}$
and at each point $P_j$,  two curves $C_\de$ and $C_{\de'}'$ are smooth and they intersect transversely.
Then we claim:
\begin{Theorem} Suppose that $P\in C\cap C'$ is bifurcated into  $\nu$
transverse intersections  in the near fibers $C_\de\cap C_{\de'}'$ as above. 
Let  $a$ and $b$  the number of positive and negative
 intersection points among $\{P_1,\dots, P_\nu\}$ ($a+b=\nu$).
Then $I_{top}(C,C';P)=a-b$.
\end{Theorem}
\begin{proof}
The assertion follows from the following  standard topological
argument.
First,  we consider 
 the   map of the pair
 $\vphi_{t,s}=(f_t,g_s)$ and its normalized one:
\[
\psi_{t,s}: S_\eps^3\to S^3 ,\quad
\psi_{t,s}(\bfz,\bar\bfz)=\vphi_{t,s}(\bfz,\bar\bfz)/ \|\vphi_{t,s}(\bfz,\bar\bfz)\|.
\]
The mapping degree of $\psi_{t,s}$
is independent of $t$ and $s$  for any $|t|\le \ga,|s|\le \ga$.

Secondly, take a sufficiently small positive number $0<r\ll \eps$
so that the disks
$B_r^4(P_j),\,j=1,\dots,\nu$ are mutually disjoint and do not intersect
with $S_\eps^3(P)$.
Then $\psi_{\de,\de'}$ is extended to a mapping 
\[\psi_{\de,\de'} : \,X:=B_\eps^4(P)\setminus\cup_{j=1}^\nu \text{Int}B_r^4(P_j)\to S^3
\] where $\text{Int} B_r^4(P_j)=B_r^4(P_j)\setminus S_r^3(P_j) $.
Thus the fundamental class $[S_\eps(P)]$ is 
equal to the sum of fundamental classes $\sum_{i=1}^\nu[S_r(P_j)]$ in $H_3(X)$,
 the mapping degree of
$\psi_{\de,\de'}: S_\eps^3(P)\to S^3$ is the sum of the local mapping
degrees
of $\psi_{\de,\de'}: S_r^3(P_j)\to S^3$.
\end{proof}
\begin{Remark} Note that $a,b$ in the above theorem  depends on the bifurcation but $a-b$
is independent of the chosen bifurcation.
Note also that $a,b$ can be $0$ which implies 
$C_\de\cap C'_{\de'}\cap B_\eps^4(P)=\emptyset$.
See Example \ref{Example1}.
\end{Remark}
\subsection{Global intersection number}
We consider the global intersection number.
Let $C:F(\bfX,\bar \bfX)=0$ and $C':\, G(\bfX,\bar \bfX)=0$
be mixed projective curves in $\BP^2$ defined by strongly polar homogeneous
polynomials
$F$ and $G$ of polar degree $d$ and $d'$ respectively.
We assume also that the mixed  singularities of $C$ and $C'$ are at worst
isolated singularities. Then by Theorem 11, \cite{MC}, they have  respective
fundamental cycles $[C]$ and $[C']$.
Here $\bfX=(X_0,X_1,X_2)$ are homogeneous coordinates of $\BP^2$.
Assume that $C\cap C'\cap \{X_0=0\}=\emptyset$.
We consider the affine space $\BC^2$ with coordinates $z_1=X_1/X_0$ and
$z_2=X_2/X_0$ respectively
and put
\[
f(\bfz,\bar\bfz):=F(1,z_1,z_2,1,\bar z_1,\bar  z_2),\quad
g(\bfz,\bar\bfz):=G(1,z_1,z_2,1,\bar z_1,\bar  z_2)
\]
respectively. Let $C\cap C'=\{P_1,\dots,P_\mu\}$.
Then by Theorem 11, \cite{MC}, the fundamental classes $[C]$, $[C']$ of
$C,C'$  exist and they satisfy, in 
$H_2(\BP^2)$    
\[
[C]=d[\BP^1],\quad [C']=d'[\BP^1]
\]
where $[\BP^1]$ is the homology class corresponding to the fundamental 
class of  the complex line $\BP^1\subset \BP^2$.
Thus we have the equality  $[C]\cdot [C']=dd'$.
Now we  have the equality:
\begin{Theorem}
\[
\sum_{j=1}^\mu \,I_{top}(C,C';P_j)=dd'.
\]
\end{Theorem}
\begin{Example}\label{Example1}
Consider the special case:
\[
 C:\,z_1=0,\quad C': g(\bfz,\bar \bfz)=2z_1+z_1\bar z_1+z_2\bar z_2=0.
\]
Then $\tilde C$ is the projective line $z_1=0$  of degree 1 and 
$\tilde C'$ is the mixed curve of polar degree 0 which is defined by
$G(\bfZ,\bar\bfZ)=2\bar Z_0 Z_1+Z_1\bar Z_1+Z_2\bar Z_2=0$.
Actually $C'$ is a 2 dimensional sphere
\[
C':  \quad y_1=0, (x_1+1)^2+x_2^2+y_2^2=1
\]
and it has  a mixed singular point  $(-1,0)$.
We see that $\tilde C\cap\tilde C'=\{(1:0:0)\}$
and $\tilde C\cdot \tilde C'=0$.
This implies $I(C,C';(0,0))=0$. In fact, consider the bifurcation
$C_t=\{z_1-t=0\}$. 
It is easy to see that $C_t\cap C'=\emptyset$ if $t>0$.
For $t<0$ small, the intersection $C_t\cap C'$ is a circle.
\end{Example}

\subsubsection{Remark}
1. {\bf Twisted line.} The singular locus of a mixed curve can be non-isolated, even if we
 assume that it does not have any real codimension 1 components.

 Consider the curve $f(\bfz,\bar \bfz)= z_1-\bar z_2$.
 Then $C$ is a
 smooth real two-plane and $\tilde C$ is defined by  
$F=Z_1\bar Z_0 -  Z_0 \bar Z_2=0$.
We call $C$  (and  $\tilde C$) {\em a twisted line}. 
Let $\bar C\subset \BP^2$ be the topological closure.
The complex line at infinity $L_\infty$ is defined by $Z_0=0$.
 To see more detail structure,  we consider the coordinate chart
$U_2=\{Z_2\ne 0\}$ with  complex coordinates
      $(u_0,u_1)=(Z_0/Z_2,Z_1/Z_2)$.
Then $\tilde C\cap U_2$ is defined by
\begin{eqnarray}\label{twisted}
f_2(u_0,u_1)=u_1\bar u_0-u_0=0.
\end{eqnarray}
We observe that

 (a)  $\tilde C=L_\infty\cup \bar C$ and $S:=L_\infty\cap \bar C$
is a circle defined by
$L_\infty\cap \{|Z_1/Z_2|=1\}$. This follows from (\ref{twisted}),
as $|u_1|=1$.

 (b) The singular locus of $\tilde C$ is  equal to
$S$. Using the coordinates $(u_0,u_1)$ on $U_2$,
$S$ is defined by $|u_1|=1$ on $L_\infty=\{u_0=0\}$.
As a  1-cycle, we orient it counterclockwise.
 Inside the circle $S$ (i.e., $|u_1|<1$),
the orientation is same with the disk
$\De:=\{u_1\in \BC||u_1|<1\}$. Outside $\{|u_1|>1\}$ of $S$, 
the orientation is
opposite to the complex structure with coordinates $u_1$. 
The singular locus can be computed by the Jacobian matrix of
$(f_{2\BR},f_{2 I})$
or by Proposition 1, \cite{OkaPolar}

 (c) Let $U_0=\{Z_0\ne 0\}$. In this coordinate, $p:C\to \BC$,
$p(z_1,z_2)=z_2$ is an orientation preserving diffeomorphism.
The circle $S_R:=\{|z_2|=R\}$ converges to $-2S$ when $R\to \infty$.
\begin{proof}
To see this, consider the large circle
$S_R$ parametrized by 
$z_1=R e^{-i\theta},\,z_2=R
 e^{i\theta},\,0\le\theta\le 2\pi$. In the chart $U_2$,
this corresponds to
\[
 u_0(\theta)=\frac 1R e^{-i\theta},\quad
u_1(\theta)=z_1/z_2=e^{-2i\theta}.
\]
\end{proof}
In \cite{MC}, we have observed that there exists a fundamental class
$[D]\in H_2(D) $ for any mixed projective curve $D$ with at most
 isolated mixed singularities. Our curve $\tilde C$ has non-isolated
 singularities along $S$. However we claim that

 \begin{Claim}$\tilde C$ has a fundamental class.
\end{Claim}
To see this, triangulate $\tilde C$ so that $S$ is a union of
 1-simplices.
Then the sum $\omega$ of all two simplices with positive orientation in $L_\infty$ satisfies
$\partial \omega=2 S$ by the observation (b). The sum $\si$ of 2 simplices in 
$\bar C$ satisfies
$\partial \si=-2S$ as we have observed in (c).
Thus $\omega+\si$ is a cycle and it gives the fundamental class.

2. It is possible that a projective mixed curve $D$
with at most isolated singularities  may have some 0-dimensional components. 
The fundamental class $[D]\in H_2(D)$ is the sum of 
2 simplices with positive orientation under  a triangulation
where singular points are vertices.

\begin{Problem} Assume that a projective mixed curve $C$ has at most  1
dimensional singular locus. Does  $C$  have always a fundamental class
as above?
\end{Problem}
\subsubsection{Remark on complex analytic cases}
Assume that $C$ and $C'$ be complex analytic curves. Assume first $P=(\al,\be)\in C\cap C'$
is a transverse intersection where $C,C'$ are non-singular.
Let $J$ be the
complex Jacobian matrix at $P$
\[
 J=\det\left(
\begin{matrix}
\frac{\partial f}{\partial z_1}(\al,\be)&
\frac{\partial f}{\partial z_2}(\al,\be)\\
\frac{\partial g}{\partial z_1}(\al,\be)&
\frac{\partial g}{\partial z_2}(\al,\be)\\
\end{matrix}\right)
\]
Then using the Cauchy-Riemann equality, we can easily show that
\[
 \det\frac{\partial(f_{\BR},f_I,g_{\BR},g_I)}{\partial  (x_1,y_1,x_2,y_2)}(\al,\be)=|J|^2>0.
\]
This implies that the local intersection number 
is 1 if the intersection is transversal at a  regular point $P$.
For a generic case, we have
\[
 I_{top}C,C';P)=\dim_{\BC}\mathcal O_P/(f,g)=I(C,C';P)\in \BN
\]
where $I(C,C';P)$ is the algebraic local intersection multiplicity
and $(f,g) $ is the ideal generated by $f,g$.

 \section{Multiplicity with sign} In this section, we
 consider the special case that 
$C: \hat f(\bfz,\bar\bfz)=0$ is a mixed curve and $C'$ is a complex   line
in $\BC^2$.
So  $\bfz=(z_1,z_2)\in \BC^2$ and  we assume that $g:=z_2$ and $\hat f|_{z_2=0}$ is a mixed polynomial of one
complex variable, $z_1$. Put $f:=\hat f|_{z_2=0}$.
Suppose that $\al\in \BC$ is an isolated mixed root of $f(z_1,\bar z_1)=0$ i.e.,
$f(\alpha,\bar \alpha)=0$ and $f(z_1,\bar z_1)\ne 0$ for any sufficiently 
near $z_1\ne \al$.
 For a positive number $\eps>0$,
we put
\[
 S_\eps^1(\al):=\{z_1\in \BC\,|\,|z_1-\al|=\eps\}.
\]
We define {\em the multiplicity with sign}
 of the root $z_1=\al$ by the mapping degree
of the normalized function
\[
 f/|f|: S^1_\eps(\al)\to S^1,\quad
\bfz\mapsto f(z_1,\bar z_1)/|f(z_1,\bar z_1).
\]
for a sufficiently small $\eps$
and we denote  the  multiplicity with sign by $\sm (f,\al)$.
The mapping degree $\sm(f,\al)$  is also called {\em the rotation number.}
We claim 
\begin{Lemma}
Let $f,\,\hat f $ be as above.  Let $g(\bfz,\bar\bfz)=z_2$.
Let $C=\{\hat f(\bfz,\bar \bfz)=0\} $ and $C'=\{z_2=0\}$.
Let $\al\in \BC$ be a root of $f$ and let $\hat \al=(\al,0)$.
Then $\hat \al\in V(\hat f,g)$ and $I_{top}(C,C';\hat \al)=\sm(f,\al)$.
\end{Lemma}
\begin{proof}
We 　use the notations: 
\[\begin{split}
&D_{\eps}(\al):=\{z\,|\, |z-\al|\le \eps\},\, S_\eps^1(\al)=\partial  D_\eps(\al),\\
&D_\eps:=D_\eps(0),\,S_{\eps}^1:=S_\eps^1(0).
\end{split}
\]
Put $f_t(\bfz,\bar \bfz)= f(z_1,\bar z_1)+t(\hat f(\bfz,\bar
 \bfz)-f(z_1,\bar z_1))$.
 Note that $f_1=\hat f,\,f_0=f$.
Take a positive number $\eps_1$ small enough so that $0$ is the unique
root of $f(z_1,\bar z_1)=0$ in $D_{\eps_1}(\al)$.
Then take $0<\eps_2\ll \eps_1$ so that $ f_t$ is non-zero on 
$S_{\eps_1}^1(\al)\times D_{\eps_2}$, that is, $f_t(\bfz,\bar \bfz)\ne 0$  if $|z_1-\al|=\eps_1,|z_2|\le \eps_2$.
For the calculation of the mapping degree of  the
normalization $\psi$ of $(\hat f_{\BR}, \hat f_I,g_{\BR},g_I)$, we can use the boundary of 
$\partial(D_{\eps_1}(\al)\times D_{\eps_2})$ in the place of the  sphere  
$S_\eps^3(\hat\al)$.
We use the Mayer-Vietoris exact sequence
of
$\partial(D_{\eps_1}(\al)\times D_{\eps_2})$
associated with the decomposition
$\{D_{\eps_1}(\al)\times S_{\eps_2}^1,\, S_{\eps_1}^1(\al)\times  D_{\eps_2})$.
Then we have the following  commutative diagram where the horizontal arrows are isomorphisms.
\[\begin{matrix}
H_3(\partial(D_{\eps_1}(\al)\times D_{\eps_2}  )&\mapright{\de}&
H_2(S_{\eps_1}^1(\al)\times S_{\eps_2}^1)\\
\mapdown{\psi_*}&&\mapdown{\psi_{*1}'}\\
H_3(\partial(D_{\eps}(\al)\times  D_{\eps_2} ) )&\mapright{\de}&
H_2(S_{\eps_1}^1(\al)\times S_{\eps_2}^1)\\
\end{matrix}
\]
The right vertical map  $\psi_{*1}'$ is 
induced by $\hat f=f_1$ and  $f_1$ is homotopic to $f_0=f$.
Therefore  $\psi_{*1}'$ coincides with
 $(\eps_1f/|f|)_*\times \id$.
The homotopy is given by the normalization of
$(f_t(\bfz,\bar\bfz), g(\bfz,\bar \bfz))$. 
Here the nomalization $\psi_*$ of $f_t(\bfz,\bar\bfz)$ is defined by  $\psi_*(\bfz,\bar\bfz)=\hat\al+(f_t(\bfz,\bar\bfz)-\hat\al)\la$ where $\la$ is the unique positive number so that  the right hand side is in
$\partial(D_{\eps}(\al)\times  D_{\eps_2} ) $.
Thus  we get 
\[\begin{split}
I_{top}(C,C';\hat \al)&=\text{mapping degree of}\,\psi_*\\
&=\text{mapping degree of}\, (\eps_1 f/|f|)_*
=\sm(f,\al).
\end{split}
\]
\end{proof}
 We define
{\em the  total multiplicity with sign}
by the sum of 
$\sm(f,\al)$ for all $\al\in V(f)$ 
where 
$V(f)=\{\al\in \BC\,|\,f(\al,\bar\al)=0\}$
and denote it  by $\SM(f)=\sum_{\al\in V(f)}\sm(f,\al)$.
Note that $\sm(f,\al)$ and    $\SM(f)$  is not necessarily positive  and it 
can be  any  integer.
\subsection{A criterion for the positivity}
Let us study some detail for a simple root $\al\in V(f)$.
First $f(z_1,\bar z_1)$ can be written as a polynomial of $w_1,\bar w_1$ 
with $w_1=z_1-\al$
by the substitution $f_\al(w_1,\bar w_1):=f(w_1+\al,\bar w_1+\bar \al)$.
Put $a:=\frac{\partial f}{\partial z_1}(\al,\bar\al)$ and 
$b:=\frac{\partial f}{\partial \bar z_1}(\al,\bar\al)$.
This implies that 
$L(w_1,\bar w_1)=a\,w_1 +b\, \bar w_1$ is the linear term of $f_\al(w_1,\bar w_1)$.
Put
\[
 a=a_1+ a_2i,\quad b=b_1+b_2i,\quad \al=\al_1+\al_2 i,\quad
a_1,a_2,b_1,b_2,\al_1,\al_2\in \BR.
\]
Then  the expansions of the real polynomials $f_{\BR},f_I$ in 
two real variables
 $(x_\al,y_\al):=(x-\al_1,y-\al_2)$ 
are given as follows:
\[\begin{split}
 f_{\BR}(x_\al,y_\al)&=\Re f(w_1+\al,\bar w_1+\bar \al)\\
& =(a_1+b_1)x_\al+(-a_2+b_2)y_\al+(\text{higher terms})\\
 f_I(x_\al,y_\al)&=\Im f(w_1+\al,\bar w_1+\bar \al)\\
& =(a_2+b_2)x_\al+(a_1-b_1)y_\al+(\text{higher terms}).
\end{split}
\]
Thus we observe that
\[\begin{split}
 \det(\frac{\partial (f_{\BR},f_I)}{\partial(x,y)}(\al_1,\al_2))&=
\left|\left(\begin{matrix} a_1+b_1&- a_2+b_2\\
a_2+b_2&a_1-b_1\end{matrix}\right)
\right|\\
&=(a_1^2+a_2^2)-(b_1^2+b_2^2)=|a|^2-|b|^2.
\end{split}
\]
We say that $\al$ is a {\em positive simple root }
 if $\al$ is a mixed-regular point for $f$
and 
$\sm(f,\al)>0$ 
 which is equivalent to
\[
  \det(\frac{\partial (f_{\BR},f_I)}{\partial(x,y)}(\al_1,\al_2))>0
\]
Similarly  $\al$ is a  {\em negative simple root}
if $\al$ is a mixed-regular point for $f$
and $\sm(f,\al)<0$. Thi sis equivalent to 
\[
 \det(\frac{\partial (f_{\BR},f_I)}{\partial(x,y)}(\al_1,\al_2))<0.
\]
Thus we get 
\begin{Proposition}\begin{enumerate}
\item
$\al$ is a positive (resp. negative) simple root if and only if
$|a|>|b|$. That is
\[\begin{split}
 &\sm(f,\al)>0\iff \left |\frac{\partial f}{\partial z_1}(\al,\bar \al)\right|>
 \left |\frac{\partial f}{\partial\bar z_1}(\al,\bar \al)\right|\\
&\sm(f,\al)<0\iff
 \left |\frac{\partial f}{\partial z_1}(\al,\bar \al)\right|<
 \left |\frac{\partial f} {\partial\bar z_1}(\al,\bar \al)\right|
 \end{split}
\]
\item
If
$\left |\frac{\partial f}{\partial z_1}(\al,\bar \al)\right|=
 \left |\frac{\partial f}{\partial\bar z_1}(\al,\bar \al)\right|$,
 $\al$ is a mixed singularity of $f$.
 \end{enumerate}
 \end{Proposition}
\subsection{Bifurcation}
Suppose that $0$ is an isolated root of  a mixed polynomial $f(u,\bar u)$.
Consider a bifurcation family  $f_t(u,\bar u)=0$ and let $\{P_1(t),\dots, P_\nu(t)\}$
be the roots of $f_t(u,\bar u)=0$ which are bifurcating from $u=0$.
Then we have
\begin{Proposition}
$\sum_{i=1}^\nu \sm(f_t,P_i(t))=\sm(f,0)$. In particular, if the roots $P_i(t)$  are simple,
$\sm(f,0)$ is equal to the difference of  the number of positive roots and the negative roots.
\end{Proposition}
The proof is similar with that of  Theorem 4. Note that $\nu$  depends on the chosen bifurcation.

\begin{Example} 
1. Let $f(u,\bar u)=u^2 \bar u$.  It is easy to see that 
$u=0$ is a non-simple singularity and $\sm(f,0)=1$.
(For a complex  polynomial singularity, $\sm(f,0)=1$ implies that $0$ is a simple root.)
Consider two bifurcation families:
\[\begin{split}
f_t(u,\bar u)=(u^2-t) \bar u,\quad g_s(u,\bar u)=u(u\bar u+s)\quad
\text{for}\,\,t,s\ge 0.
\end{split}
\]
Note that $f_t=0$ has two positive roots $u=\pm \sqrt t$ and a negative root $u=0$.
$g_s=0$ has only one positive root  $u=0$ for $s>0$.

\begin{Assertion} \label{assertion}
Let $f(u,\bar u)=u^n+u+\bar u$ for any $n\ge 2$.  Then 
\[
 \sm(f,0)=\begin{cases}  1\quad & n:\,  \text{even}\\
                       -1\quad &n\equiv 3\mod 4\\
                        1\quad &n\equiv 1\mod 4
\end{cases}
\]
\end{Assertion}
\end{Example}
For the proof, show the Appendix (\S \ref{Appendix} ).

\subsection{Admissible mixed polynomial and the main theorem}
We consider a mixed polynomial 
$f(u,\bar u)=\sum_{\nu,\mu}c_{\nu,\,u}
u^\nu\bar u^\mu$ of one variable $u$.
The {\em  maximal  degree} of $f$ is defined by 
$\bar d=\max\{\nu+\mu\,|\, c_{\nu,\mu\ne 0}\}$. 
We denote $\bar d=\bar d(f)$.
Similarly we define the {\em minimal degree} of $f$ at the origin
by ${\underline d}:=\min\,\{\nu+\mu\,|\, c_{\nu,\mu\ne 0}\}$.
and we denote ${\underline d}=\underline{d}(f)$.
Note that the minimal degree is a local invariant but the maximal degree
is a global invariant. That is, $\bar d(f)$ is invariant under the
parallel
change of coordinate $v=u-a$.
%
For a positive integer $\ell$,  we put
  \[
  f_\ell(u,\bar u):=\sum_{\nu+\mu=\ell}c_{\nu,\mu}u^\nu\bar u^\mu.
  \]
Then we can write 
\[\begin{split}
 f(u,\bar u)&=f_{\bar d}(u,\bar u)+f_{\bar d-1}(u,\bar u)+\cdots+
f_{\underline d+1}(u,\bar u)+f_{\underline d}(u,\bar u)\\
&=f_{\bar d}(u,\bar u)+k(u,\bar u),\quad\\
&=f_{\underline d}(u,\bar u)+j(u,\bar u)
\end{split}
\]
with $\bar( k)<\bar d$ and $\underline{d}(j)>\underline d$. Note that we have  a unique 
 factorization of $f_{\bar d}$ and $f_{\underline d}$ as follows.
\begin{eqnarray}
 f_{\bar d}(u,\bar u)&=c u^{ p}\bar u^{ q}\prod_{j=1}^{s}(u+\ga_j \bar
 u)^{\nu_j},\quad
p+q+\sum_{j=1}^{s}\nu_j=\bar d,\,c\in \BC^*\label{fac1}\\
f_{\underline d}(u,\bar u)&=c' u^{a}\bar u^{b}
\prod_{j=1}^{s'}(u+{\de}_j {\bar  u})^{{\mu}_j},\quad
{a}+{b}+\sum_{j=1}^{s'}{\mu}_j={\underline d},\,c'\in \BC^*\label{fac2}
\end{eqnarray}
where $\ga_1,\dots, \ga_{s}$ (respectively 
${\de}_1,\dots, {\de}_{s'}$)
are mutually distinct non-zero complex numbers.
We say that $f$ is {\em admissible at infinity}
(respectively {\em admissible at the origin}) if
$|\ga_j|\ne 1$ for $j=1,\dots, s$ (resp.  $|{\de}_j|\ne 1,j=1,\dots, {s'}$). 
For non-zero complex number $\xi$, we put 
\[\eps(\xi)=\begin{cases}
1\quad &|\xi|<1\\0\quad& |\xi|=1\\-1\quad &|\xi|>1\end{cases}
\]
and we consider the following  integers:
\[
\be(f):= p- q+\sum_{j=1}^{s} \eps(\ga_j)\nu_j,\quad
\rho(f,0):= {a}-{b}+\sum_{j=1}^{s'} {\eps}(\de_j)\mu_j.
\]

Our main result is the following.
\begin{Theorem} \label{main}
\begin{enumerate}
\item
Assume that $f(u,\bar u)$ be an admissible mixed polynomial
 at infinity.
 Then
$\SM(f)=\be(f)$.
\item
Assume that $f(u,\bar u)$ be an admissible mixed polynomial
 at the origin.
 Then
$\sm(f,0)=\rho(f,0)$.
\end{enumerate}
\end{Theorem}

\begin{proof}
Put $\bar d=\bar d(f)$ and assume that $f_{\bar d}$ is factored as in (\ref{fac1}).
In the case
$s=0$, the proof is the same  with that  of Theorem 11, \cite{MC}.
In the general case, we first assume  that 
\[
 |\ga_1|\le \cdots\le |\ga_\ell|<1<|\ga_{\ell+1}|\le\dots\le |\ga_s|.
\]
Let $R$ be a positive number. 
First we observe that for any $u\in S_R^1$,
\[\begin{split}
 |f_{\bar d}(u,\bar u)|&=|c| R^{\bar d} \prod_{j=1}^\ell|1+\ga_j\bar u/u|^{\nu_j}
\prod_{j=\ell+1}^s|u/\bar u+\ga_j|^{\nu_j}\\
&\ge |c| R^{\bar d} \prod_{j=1}^\ell(1-|\ga_j|)^{\nu_j}
\prod_{j=\ell+1}^{s}(|\ga_j|-1)^{\nu_j}\\
&\ge MR^{\bar d}
\end{split}
\]
for some positive constant $M>0$.
We can choose a sufficiently large $R>0$ so that 
\[
 |f_{\bar d}(u,\bar u)|> 2|k(u,\bar u)|,\quad \forall u,\, |u|\ge R.
\]
The rest of the argument is  exactly  same  as the proof of
Theorem 11, \cite{MC}.
Let $V(f)=\{\al_1,\dots, \al_m\}$  and take a small positive number 
$\eps$ so that $D_\eps(\al_j)\cap V(f)=\{\al_j\}$ where 
$D_\eps(a):=\{u\,|\, |u-a|\le \eps\}$.
First, as $f/|f|:S_R^1\to S^1$ is extended to $D_R(O)\setminus
 \bigcup_{i=1}^m D_\eps(\al_j)$,
we have 
\[
 \text{mapping degree }
(f/|f|: S_R^1\to S^1 ) =\sum_{j=1}^m\sm(f,\al_j).
\]
To compute the mapping degree $f/|f|: S_R^1\to S^1$,
we consider the family of polynomials
$f(u,\bar u,t):=f_{\bar d}(u,\bar u)+(1-t)k(u,\bar u)$.
This family is non-vanishing on $S_R^1$.
Note that $f(u,\bar u,0)=f(u,\bar u)$ and $f(u,\bar u,1)=f_{\bar d}(u,\bar u)$.
As $f/|f|\simeq f_{\bar d}/|f_{\bar d}|$ on $S_R^1$, we have
\[\begin{split}
 \sum_{j=1}^m\sm(f,\al_j)&=\text{mapping degree of }
\, f/|f|: S_R^1\to S^1\\
&= \text{mapping degree of } \, f_{\bar d}/|f_{\bar d}|.
\end{split}
\]
Now we will show that  the mapping degree of $f_{\bar d}/|f_{\bar d}|$ is  equal to the
 integer $\be(f)$. For this purpose, 
we write $f_{\bar d}$ as
\[
 \begin{split}
f_{\bar d}(u,\bar u)&=u^{\hat p}
\bar u^{\hat q}
\prod_{j=1}^\ell(1+\ga_j\frac{\bar u}u)^{\nu_j}
\prod_{k=\ell+1}^s(\frac{ u}{\bar u}+\ga_j)^{\nu_k}\\\text{where}\,\,
\hat p&= p+\sum_{j=1}^\ell \nu_j,\quad 
\hat q= q+\sum_{j=\ell+1}^s \nu_j.
\end{split}
\]
Note that 
\[
 \be(f)=\hat p- \hat q= p- q+\sum_{j=1}^\ell \nu_j-\sum_{j=\ell+1}^{s} \nu_j
\]
 in the above notation.
We observe that 
\[\begin{split}
& 1+\ga_j\frac{\bar u}u\in D_{|\ga_j|}(1),\quad 1\le j\le \ell,\,\,u\in S_R^1\\
&\frac u{\bar u}+\ga_k\in D_1(\ga_k),\quad \ell+1\le k\le s,\,\,u\in S_R^1
\end{split}
\]
where
$D_\eps(\eta)=\{\zeta\in \BC\,|\,|\zeta-\eta|<\eps\}$.
It is easy to observe that
\[
0\notin  D_{|\ga_j|}(1)\,(j\le \ell),\,\quad 0\notin  D_{1}(\ga_k)\, (k\ge \ell+1).
\]
Consider the family of polynomials
\[
  f_{\bar d}(u,\bar u,t):=u^{\hat p}\bar u^{\hat q}\prod_{j=1}^{s}(1+t \ga_j\frac{\bar u}u)^{\nu_j}
\prod_{k=s+1}^{s}(t\frac{ u}{\bar u}+ \ga_j)^{\nu_k},\,0\le t\le 1.
\]
Note that $ f_{\bar d}(u,\bar u,1)=f_{\bar d}(u,\bar u)$ and
$ f_{\bar d}(u,\bar u,0)=u^{\hat p}\bar u^{\hat q}$.
As $ f_{\bar d}(u,\bar u,t),\,0\le t\le 1$ give a homotopy on $S_R^1$, the assertion follows
 from the fact that the mapping degree of 
$ u^{\hat p}\bar u^{\hat q}   $ is $\be(f)$. This proves the first assertion (1).

The second assertion (2) is proved by the same argument:
\nl
--Take a sufficiently small $r>0$ so that 
\[|f_{\underline d}(u,\bar u)|\ge 2 |j(u,\bar u)|, \quad \forall  u,\,|u|\le r
\]
where $f=f_{\underline d}+j$.
\nl
--Observe that the homotopy
$\underline{f}(u,\bar u,t)=f_{\underline d}(u,\bar u)+t j(u,\bar u),\,0\le t\le 1$
is non-vanishing on the circle $S_{r}^1$.
\nl
--The normalization  $f_{\underline d}/|f_{\underline d}|$ 
of $f_{\underline d}(u,\bar u,0)$ is homotopic to  that of
$u^{\rho(f,0)}$.
\end{proof}
\subsection{Compactification}
Suppose that we are given a mixed polynomial
$f(u,\bar u)=\sum_{\nu,\mu} c_{\nu,\mu}\,u^\nu\bar u^\mu$.
Let $\bar d=\overline{\deg}\, f$ and put
\[
 d_+=\max\,\{\nu\,|\, c_{\nu,\mu}\ne 0\},\,\,
 d_-=\max\,\{\mu\,|\, c_{\nu,\mu}\ne 0\}.
\]
Define
\[
 F(z_0,z_1,\bar z_0,\bar z_1):=z_0^{d_+}\bar z_0^{d_-}
f(z_1/z_0, \bar z_1/\bar z_0).
\]
$F(z_0,z_1,\bar z_0,\bar z_1)$ is the mixed
homogenization defined in \S 1.
Put $d_{h}=d_++d_-$ and
 $q_{h}=d_+-d_-$.
By the definition, we have the following assertion.
\begin{Proposition} Assume that $f_{\bar d}(u,\bar u)$  be factorized as (2) and let 
 $F(z_0,z_1,\bar z_0,\bar z_1)$ be as above.
$F$ is a strongly polar homogeneous polynomial of 
radial degree $d_{h}$ and  polar degree $q_{h}$ and we have the inequality
 $d_{h}\ge\bar d= {\overline{\deg} }\,f$.
\begin{enumerate}
\item
The equality  $d_{h}=\bar d$ holds if and only if
\[p=d_+,\,\,q=d_-,\, s=0.
\]
\item Assume that $d_{h}>\bar d$. Then $(0:1)\in V(F)$. Namely each monomial
in $F(z_0,z_1,\bar z_0,\bar z_1)$ contains either $z_0$ or $\bar z_0$.
\end{enumerate}
\end{Proposition}

\subsubsection{Generic line at infinity and a generic affine chart } 
Let $F(\bfz)$ be a strongly polar homogeneous polynomial of two variables 
$\bfz=(z_0,z_1)$ of radial and polar degree $d$ and $q$. We can write
$2r=d-q$
for some integer $r\ge 0$.
If the variable $z_0$ is generic (i.e., there is a monomial which does
not contain $z_0$ and $\bar z_0$), in the affine coordinate $U_0$,
$V(F)\cap U_0$ is defined by  $f(u,\bar u):=F(1,u,1,\bar u)$
and we can write
\[
 f_d(u,\bar u)=c u^{q+r}\bar u^r,\,\, f=
 f_d+\text{(lower terms)},\,\,u=z_1/z_0.
\]
In this case, we have shown that $\SM(f)=q$ in Theorem 11,
\cite{MC}.
Thus Theorem 11 \cite{MC} is a special case of Theorem \ref{main}.

\subsubsection{Example} Consider the  polynomial:
\[
 f(u,\bar u)=u^2\bar u(u-2\bar u)+1.
\]
$V(k)$ consists of 4 points
\[
 u=\pm \root{4}\of {1/3} i,\,\, \pm 1.
\]
The  multiplicities with sign of the first two roots
$\{\pm \root{4}\of {1/3} i\}$ are 1 and the latter two 
roots $\{\pm 1\}$ are -1. This implies that 
$\SM(f)=0$ as Theorem \ref{main} asserts.

The mixed  homogenization $f(\bfz,\bar \bfz)$ is given by
\[
 F(\bfz,\bar \bfz)=z_1^2\bar z_1(z_1\bar z_0-2 \bar z_1 z_0)+ z_0^3\bar z_0^2.
\]
We see that $f(\bfz,\bar \bfz)$ is a strongly polar homogeneous
polynomial
of radial and polar degrees $5$ and 1 respectively.
We observe that $(0:1)$ is on $V(f)$
and it has multiplicity with sign 1.  Now take the generic affine coordinate chart
$U_1:=\{z_1\ne 0\}$ with the coordinate $v=z_0/z_1$.
Then the affine equation of $V(f)\cap U_1$ is given as

\[
 f'(v)=\bar v-2 v+v^3\bar v^2
\]
with $5$ points.
 Note that $\sm(f',0)=1$.

\subsubsection{Appendix: Proof of Assertion
 }\label{Appendix}
Recall $f(u,\bar u)=u^n+u+\bar u$.
The proof follows the following observations.

1. $\SM(f)=n$ by Theorem \ref{main}.

2. For any $\al\in V(f)\setminus \{0\}$, $\al$ is a simple mixed root
 with
$\sm(f,\al)=1$.

3. The number, say $\beta$,  of non-zero mixed roots  
 of $f$ is given as follows:
\[
 \beta=\begin{cases}
n-1\quad &n\, \text{even}\\
n+1\quad &n\equiv 3\mod 4\\
n-1\quad &n\equiv 1\mod 4\\
\end{cases}
\]
Let us show the observation 2. So assume that $\al\in V(f)$ and $\al\ne
 0$.
Take the coordinate $v:=u-\al$.
Then
\[\begin{split}
 f(v+\al)&=\al^n+\al+\bar\al+n\al^{n-1}v+v+\bar v+\text{(higher
   terms in $ v)$}\\
&=(n\al^{n-1}+1)v+\bar v+\text{(higher terms in  $v$)}\\
&=(-(n-1)-n\frac{\bar \al}{\al}) v+\bar v+\text{(higher terms in $v$)}\\
\end{split}
\]
Now we conclude the assertion by Proposition 9 as
\[
|(n-1)-n\frac{\bar \al}{\al}|\ge n-(n-1)=1
\] and by the equality
takes place if and only if 
$\bar \al=-\al$, that is $\al$ is purely imaginary. This does not happen by the following calculation.

Now we show the observation 3. As the calculation is easy,
 we only show the result. 
Assume  $f(u)=0$ with $u\ne 0$.
Put  $u=r\exp(i a), \,0\le a< 2\pi$ in the polar coordinates.
Then we have
\[
 r^n\sin(na)=0,\quad r^{n}\cos(na)+2r\cos (a)=0.
\]
Thus the first equality says that 
\[
 na=j\pi,\,\,j=0,\dots, 2n-1
\]
The second equality has a positive solution
for $r$ if and only if
$\cos(na)\cos (a)<0$. This implies that $\al$ is not a pure imaginary complex number.
Assume  $n=4k$ for example. Then the solution exists for the following.
\[\begin{split}
  \frac{a}{\pi}&=
\{1,3,\dots, 2k-1,2k+2,2k+4,\dots, 6k-2,6k+1,\dots, 8k-1\}\\
\beta&=4k-1,\,\,\sm(f,0)=4k-\beta=1\end{split}
\]
For the case $n=4k+2$,
\[\begin{split}
  \frac{a}{\pi}&=
\{1,3,\dots, 2k-1,2k+2,2k+4,\dots, 6k+2,6k+5,\dots, 8k+3\}\\
\beta&=4k+1,\,\,\sm(f,0)=4k+2-\beta=1\end{split}
\]
For the case $n=4k-1$, we have
\[\begin{split}
  \frac{a}{\pi}&=
\{1,3,\dots, 2k-1,2k,2k+2,\dots, 6k-2,6k-1,\dots, 8k-3\}\\
\beta&=4k,\,\,\sm(f,0)=4k-1-\beta=-1\end{split}
\]
For $n=4k+1$, we have 
\[\begin{split}
  \frac{a}{\pi}&=
\{1,3,\dots, 2k-1,2k+2,2k+4,\dots, 6k,6k+3,\dots, 8k+1\}\\
\beta&=4k,\,\,\sm(f,0)=4k+1-\beta=1\end{split}.
\]
\subsection{Figure}Let us consider the case $n=2$, $f(u)=u^2+u+\bar u$.
Note that $f(u)$ has two mixed singular points, $O$ and
$P=(-2,0)$.

The following figures shows  the trace of $f(u(\theta),\bar u(\theta))$,
$u(\theta)=r\exp(i\theta),\,0\le \theta\le 2\pi$ for $r=3/2,2,3$
respectively.

\nin
Case $r=3/2$: The Figure 1 shows that $\sm(f,0)=1$.

\vspace{1cm}
\begin{figure}[here]
{\includegraphics[width=5cm,height=5cm]{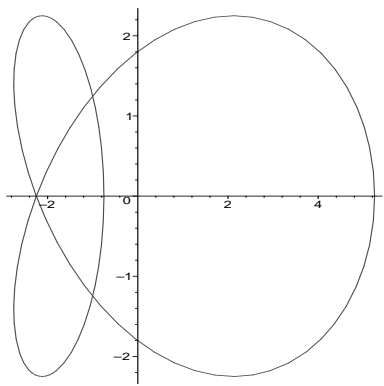}}
\caption{\label{r1}$n=2,\,r=3/2$}
\end{figure}

\nin
Case $r=2$: Figure corresponds the critical case that $|u|=2$ passes
through the mixed singular point $(-2,0)$.
\begin{figure}[here]
\includegraphics[width=5cm,height=5cm]{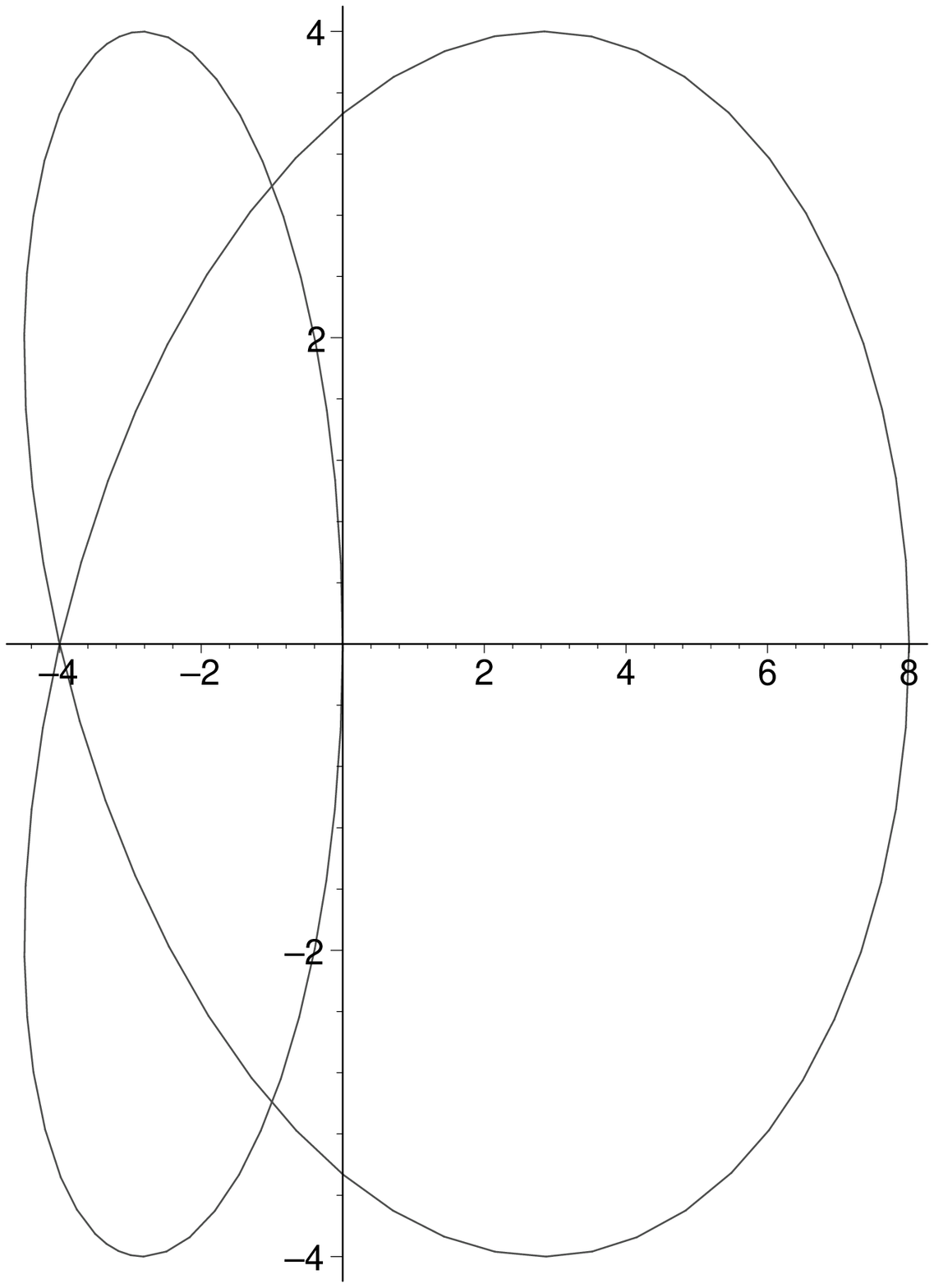}
\caption{\label{r2}$n=2,\, r=2$}
\end{figure}

\newpage\nin
Case $r=3$. The disk $|u|\le 3$ contains mixed singular point $(-2,0)$
and
$\SM(f)=2$.
\begin{figure}[here]
\includegraphics[width=5cm,height=5cm]{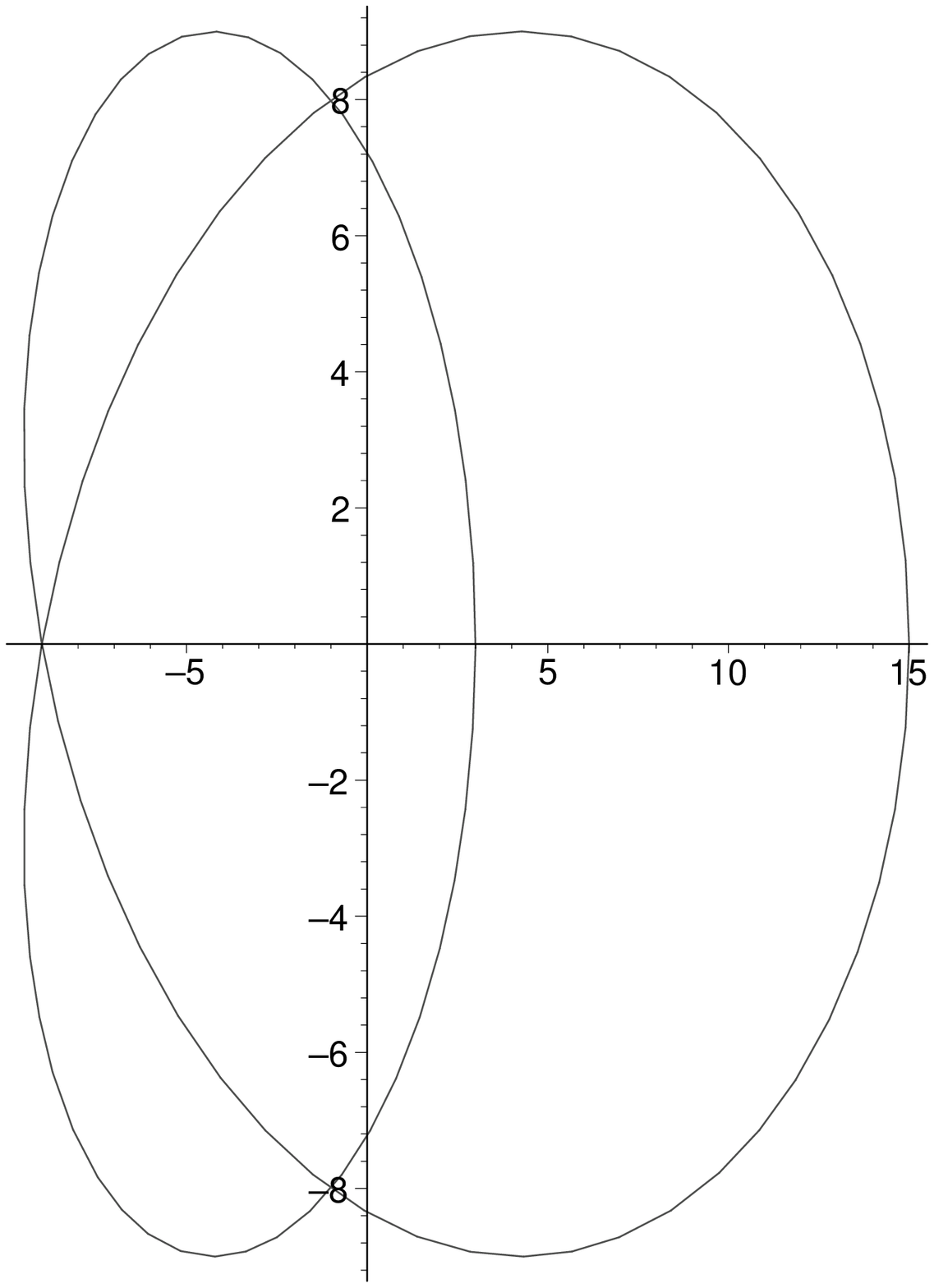}
\caption{\label{r3}$n=2,\,r=3$}
\end{figure}

\def\cprime{$'$} \def\cprime{$'$} \def\cprime{$'$} \def\cprime{$'$}
  \def\cprime{$'$} \def\cprime{$'$} \def\cprime{$'$} \def\cprime{$'$}

\end{document}